\documentclass[12pt]{amsart}

\usepackage[utf8]{inputenc}
\usepackage[margin=1.25in]{geometry}
\usepackage{textcomp}
\usepackage{amsfonts}
\usepackage{mathtools}
\usepackage{amssymb}
\usepackage{tikz-cd}
\usepackage{multicol}
\usepackage{ wasysym }
\usepackage{mathrsfs}
\usepackage{graphicx}
\graphicspath{ {./images/} }
\usepackage{amsmath}
\usepackage{float}
\usepackage{amsthm}

\usepackage[nonamebreak,square]{natbib}
\setcitestyle{numbers, open={[}, close={]}}
\DeclareMathOperator{\Z}{\mathbb{Z}}
\DeclareMathOperator{\R}{\mathbb{R}}
\DeclareMathOperator{\N}{\mathbb{N}}

\DeclareMathOperator{\Q}{\mathbb{Q}}

\DeclareMathOperator{\Hyp}{\mathbb{H}}

\newcommand{\sm}[1]{\colon #1\rightarrow #1}
\newcommand{\hb}{\partial_\infty\Hyp^2}

\newcommand{\lfol}{\mathcal{W}^{u/s}}

\newcommand{\hty}[2]{#1\colon #2\times I\rightarrow #2}

\theoremstyle{plain}
\newtheorem{theorem}{\bf Theorem}[section]
\newtheorem{prop}[theorem]{\bf{Proposition}}
\newtheorem{lemma}[theorem]{\bf{Lemma}}
\newtheorem{corollary}[theorem]{\bf{Corollary}}
\newtheorem{thmx}{Theorem}
\newtheorem{corx}[thmx]{Corollary}

\theoremstyle{definition}
\newtheorem{mydef}[theorem]{\bf Definition}

\title{Periodic points of endperiodic maps}
\author{Ellis Buckminster}
\date{}

\begin{document}
\begin{abstract} Let $g\sm{L}$ be an atoroidal, endperiodic map on an infinite type surface $L$ with no boundary and finitely many ends, each of which is accumulated by genus. By work of Landry, Minsky, and Taylor in \cite{endperiodic}, $g$ is isotopic to a \emph{spun pseudo-Anosov} map $f$. We show that spun pseudo-Anosov maps minimize the number of periodic points of period $n$ for sufficiently high $n$ over all maps in their homotopy class, strengthening Theorem 6.1 of \cite{endperiodic}. We also show that the same theorem holds for atoroidal Handel--Miller maps when one only considers periodic points that lie in the intersection of the stable and unstable laminations. Furthermore, we show via example that spun-pseudo Anosov and Handel--Miller maps do not always minimize the number of periodic points of low period.
    
\end{abstract}

\maketitle
\section{Introduction}\label{sec:introduction}

Let $g\sm{L}$ be an endperiodic map on an infinite type surface $L$ with no boundary and finitely many ends, each of which is accumulated by genus. 
In \cite{endperiodic}, Landry, Minsky, and Taylor define spun pseudo-Anosov maps on such surfaces $L$ which by definition preserve a pair of transverse singular foliations, making them a type of generalization of pseudo-Anosov maps to the infinite type setting (see Definition \ref{def:spa}). They show that such a map $g$ is isotopic to a spun pseudo-Anosov map $f$ if and only if $g$ is atoroidal, i.e. it doesn't fix an essential multicurve up to isotopy (Theorem A of \cite{endperiodic}). Thus, spun pseudo-Anosov maps provide a type of normal form for and offer a way of studying such endperiodic maps.

Spun pseudo-Anosov (spA) maps share certain dynamical features with pseudo-Anosovs on finite type surfaces. Thurston showed that for every $n$, pseudo-Anosov homeomorphisms minimize the number of periodic points of period $n$ over all maps in their homotopy class (\cite{thurston1988geometry}). Theorem 6.1 of \cite{endperiodic} states that spA maps minimize points of period $n$ among all homotopic \emph{endperiodic} maps for all $n$. Our main theorem is that for large enough $n$, this holds for all maps homotopic to a spA map, not just the endperiodic ones.

\begin{thmx}
     Let $f\sm{L}$ be a spun pseudo-Anosov map. For some $N\in \N$, every homeomorphism $g$ homotopic to $f$ has no fewer points of period $n$ than $f$ for all $n> N$.
\end{thmx}

The constant $N$ is the highest period of a periodic point lying on an escaping half leaf (see Section \ref{subsec:spa} for definitions). The result is sharp in that one cannot in general remove the constant $N$ -- there exist spun pseudo-Anosovs maps with fixed points and a homotopic homeomorphism without fixed points, see Section \ref{sec:ex} for an example.

Defining the \emph{growth rate} $\lambda(g)$ of an endperiodic map $g$ by
\[\lambda(g)=\lim\sup_{n\rightarrow\infty}\sqrt[n]{\# \text{Fix}(g^n)},\] we obtain the following corollary, generalizing Corollary 6.2 of \cite{endperiodic}.

\begin{corx}
      Let $f\colon L\rightarrow L$ be spA.  Then \[\lambda(f)=\inf_{g\simeq f}\lambda(g), \] where the infimum is taken over all maps $g$ homotopic to $f$.
\end{corx}

Such a map $g$ is also isotopic to a Handel--Miller representative originally by work of Handel and Miller and later written up and further studied by Cantwell, Conlon, and Fenley (see Theorem 4.54 of \cite{cantwell2021endperiodic}). These are maps that preserve a pair of laminations $\Lambda^+_{HM}$ and $\Lambda^-_{HM}$, called the Handel--Miller laminations. We let $\mathcal{K}=\Lambda^+_{HM}\cap\Lambda^-_{HM}$, and call $h\sm{\mathcal{K}}$ the \emph{core dynamical system} of $h$ after Cantwell, Conlon, and Fenley.  We also obtain a version of Theorem A for atoroidal Handel--Miller maps.

\begin{thmx}
    Let $h\sm{L}$ be an atoroidal Handel--Miller map with stable and unstable Handel--Miller laminations $\Lambda^+_{HM}$ and $\Lambda^-_{HM}$. For some $N\in\N$, every $g\simeq h$ has no fewer points of period $n$ than the core dynamical system of $h$ for all $n>N$.
\end{thmx}

Note that the fact that we only obtain a statement on the number of periodic points for the core dynamical system of $h$ is unsurprising, as Handel--Miller maps are only uniquely defined on $\mathcal{K}$. Handel--Miller maps and spun pseudo-Anosov maps are closely related (see Section 8 of \cite{endperiodic}). As with Theorem  A, the constant $N$ cannot in general be removed from the theorem statement.\\

The proof of Theorem 6.1 in \cite{endperiodic} (Theorem \ref{thm:lmt} in this paper) goes as follows. Given an endperiodic map $g$ homotopic to a spA map $f$, the authors find an injective map from periodic points of period $n$ of $g$ to Nielsen equivalent points of $f$ of the same period. They do this via showing that points on $\hb$ fixed by a lift $\widetilde{g}$ of $g$ are sources or sinks for the entire action of $\widetilde{g}$ on $\overline{\Hyp^2}$; this is where the endperiodic assumption on $g$ comes in. They then conclude via the Lefschetz-Hopf theorem that $\widetilde{g}$ must have a fixed point in $\Hyp^2$, proving Theorem \ref{thm:lmt} below. We will use a similar argument here, but the argument in \cite{endperiodic} that a fixed point of $\widetilde{g}$ on $\hb$ is a sink or source for the action on $\overline{\Hyp^2}$ does not go through. Instead, in Lemma \ref{lem:index} we show that $\widetilde{g}$ is homotopic through maps with no new fixed points to a map that has sink or source dynamics in a neighborhood of a fixed point $z\in\hb$ of $\widetilde{g}$. We do this by taking inspiration from work of Handel and Thurston (see Lemma 3.1 of \cite{handelthurston} in particular), though we must modify their arguments to the setting of an infinite type surface.

\subsection{Acknowledgments}\label{subsec:ack} I would like to thank Sam Taylor for many helpful conversations about this problem. I would also like to thank Dan Margalit and the referee for feedback on earlier drafts of this paper.

\section{Background}\label{sec:background}

\subsection{Endperiodic homeomorphisms}\label{subsec:endperiodic}

Throughout, let $L$ be an infinite type surface with no boundary and finitely many ends, each of which is accumulated by genus. We can put a hyperbolic metric on such a surface and identify the universal cover of $L$ with $\Hyp^2$. We may further assume that this metric is \emph{standard}, i.e. doesn't contain any embedded hyperbolic half planes. Let $\pi\colon \Hyp^2\rightarrow L$ be a covering map, and identify the group of deck transformations with $\pi_1(L)$. Any homeomorphism $h\colon L\rightarrow L$ can be lifted to $\Hyp^2$ and extended to $\overline{\Hyp^2}=\Hyp^2\cup\partial_\infty\Hyp^2$ by work of Cantwell and Conlon (see Theorem 2 of \cite{cantwell2015hyperbolic}). If $f,g\sm{L}$ are homotopic maps with compatible lifts $\widetilde{f},\widetilde{g}\sm{\overline{\Hyp^2}}$ (i.e. the lifts come from lifting a homotopy), then $\widetilde{f}$ and $\widetilde{g}$ agree on $\hb$ (see Corollary 5 of \cite{cantwell2015hyperbolic}).

In proving Lemma \ref{lem:K}, we will need the following lemma.

\begin{lemma}\label{lem:commute} Let $g\sm{L}$ be atoroidal, and let $\widetilde{g}\sm{\overline{\Hyp^2}}$ be a lift of $g$. Then $\widetilde{g}$ commutes with no nontrivial elements of $\pi_1(L)$. 
\end{lemma}

\begin{proof}
 Suppose $\widetilde{g}$ commuted with some nontrivial $\gamma\in\pi_1(L)$.  As $L$ has no cusps, every element in $\pi_1(L)$ acts of $\overline{\Hyp^2}$ with an axis of translation, so $\gamma$ has some axis of translation $A_\gamma$. 
 Then \(\gamma\widetilde{g}A_\gamma=\widetilde{g}\gamma A_\gamma=\widetilde{g}A_\gamma,\) so $\gamma$ fixes the geodesic $\widetilde{g}A_\gamma$. Since nontrivial isometries of $\Hyp^2$ can fix at most one geodesic, we must have $\widetilde{g}A_\gamma=A_\gamma$. But then $g$ fixes the closed geodesic $\pi(A_\gamma)$, contradicting the fact that $g$ is atoroidal.
\end{proof}

Let $g\sm{L}$ be a homeomorphism, and let $\mathcal{E}$ be an end of $L$. The end $\mathcal{E}$ is called \emph{attracting} under $g$ if for some neighborhood $U_\mathcal{E}$ of $\mathcal{E}$, positive iterates of some power of $g$ applied to $U_\mathcal{E}$ escape the end $\mathcal{E}$, i.e. for some $n$, $g^n(U_\mathcal{E})\subseteq U_\mathcal{E}$ and $\bigcap_{i\geq 1}g^{ni}(U_\mathcal{E})=\varnothing$. An end is called \emph{repelling} if it is attracting under $g^{-1}$. The map $g$ is called \emph{endperiodic} if every end of $L$ is either attracting or repelling under $g$.  Fix such a neighborhood $U_\mathcal{E}$ of each end $\mathcal{E}$ such that the $U_\mathcal{E}$ are pairwise disjoint. Let $U_+$ be the union of $U_\mathcal{E}$ over all attracting ends, and $U_-$ the union of $U_\mathcal{E}$ over all repelling ends. Define the \emph{positive} and \emph{negative escaping sets} of $g$ by 
\[\mathcal{U}_+=\bigcup_{n\geq 0}g^{-n}(U_+)\] and  \[\mathcal{U}_-=\bigcup_{n\geq 0}g^n(U_-)\] respectively. 

A $g$\emph{-juncture} of an end $\mathcal{E}$ is a compact 1-manifold that bounds such a neighborhood $U_\mathcal{E}$ of $\mathcal{E}$ as above. A $g$-juncture is \emph{positive} (resp. \emph{negative}) if the corresponding end is positive (resp. negative). A \emph{tightened} $g$\emph{-juncture} is the geodesic representative of a $g$-juncture. For our purposes, we may pick each $g$-juncture to be a single closed geodesic; fix such a choice. 

Given an endperiodic map $g\sm{L}$, we can form the \emph{mapping torus} $N_g=L\times I/(x,1)\sim(g(x),0)$, a non-compact 3-manifold.  Following Fenley, we can then create the \emph{compactified mapping torus} $\overline{N_g}=N_g\cup\partial_+N_g\cup\partial_-N_g$ as follows. Add an ideal point to each flow line in $\mathcal{U}_+\times I/\sim\;\subseteq N_g$; call this set of ideal points $\partial_+N_g$. Define $\partial_-N_g$ similarly. This defines a depth 1 foliation on $N_g$, where the depth 1 leaves are all homeomorphic to $L$, make up the set $N_g\setminus\{\partial_+N_g\cup\partial_-N_g\}$, and spiral around the depth 0 leaves, which are compact and make up the set $\partial_+N_g\cup\partial_-N_g$. For more details on this construction, see \cite{fenley}.

Another way to view this construction, following the exposition of Field, Kim, Leininger, and Loving in \cite{field2023end}, is as follows. As $N_g$ is a fiber bundle over the circle with fiber $L$, we have a covering map $p\colon L\times(-\infty,\infty)\rightarrow N_f$ coming from unwinding the circle base. Define $\widetilde{N_g}=L\times (-\infty,\infty)\cup\mathcal{U}_+\times\{\infty\}\cup\mathcal{U}_-\times\{-\infty\}$. The map $G\sm{\widetilde{N_g}}$ defined by $G(x,t)=(g(x),t-1)$, where $\pm\infty-1=\pm\infty$, is a homeomorphism, and the cyclic group $\langle G\rangle$ acts on $\widetilde{N_g}$ properly discontinuously and cocompactly. Thus, the quotient $\overline{N_g}=\widetilde{N_g}/\langle G\rangle$ is a compact manifold, with $N_g=L\times(-\infty,\infty)/\langle G\rangle$. Its boundary is $\partial_+N_g\cup\partial_-N_g$, where $\partial_+N_g=\mathcal{U}_+/\langle G\rangle$ and $\partial_-N_g=\mathcal{U}_-/\langle G\rangle$.

\subsection{Spun pseudo-Anosov maps}\label{subsec:spa} 

Before defining spun pseudo-Anosov maps, we record the main property of these maps proven in \cite{endperiodic} that we will strengthen. 

\begin{theorem}\label{thm:lmt} (Theorem 6.1 of \cite{endperiodic}). If $f\sm{L}$ is spA, then for each $n\geq1$, the homeomorphism $f$ has the minimum number of periodic points of period $n$ among all homotopic, endperiodic homeomorphisms.
\end{theorem}

In \cite{endperiodic}, a spun pseudo-Anosov map $f\sm{L}$ is defined as follows.

\begin{mydef}\label{def:spa} (Definition 3.1 of \cite{endperiodic}).
    An endperiodic map $f\sm{L}$ is \emph{spun pseudo-Anosov} (or \emph{spA}) if there exists a depth one foliation $\mathcal{F}$ of a hyperbolic 3-manifold $M$ and a transverse, circular almost pseudo-Anosov flow $\varphi$ that is minimally blown up with respect to $\mathcal{F}$ such that $L$ is a leaf of $\mathcal{F}$ and $f$ is a power of the first return map induced by $\varphi$.
\end{mydef}

 The pseudo-Anosov flow $\varphi$ being \emph{circular} means it is the suspension flow of a pseudo-Anosov homeomorphism, i.e. for a pseudo-Anosov homeomorphism $h\sm{S}$ on a surface $S$, $\varphi$ is the image of the vertical flow $\psi$ on $S\times I$ under the quotient map $S\times I/(x,1)\sim(h(x),0)$. A \emph{depth one foliation} is a foliation such that the complement of the compact leaves is a circle bundle.
 
 The definitions of \emph{almost} pseudo-Anosov flows and being \emph{minimally blown up} with respect to the foliation $\mathcal{F}$ involve the concept of dynamic blowups, defined by Mosher in \cite{mosher1990correction}. The interested reader can find a discussion of dynamic blowups in this context in Section 3 of \cite{transverse}. They are irrelevant to our purposes, as Theorem 4.2 of \cite{endperiodic} states every spA map is isotopic to an spA$^+$ map, as defined below.

 \begin{mydef}\label{def:spa+} (Definition 3.4 of \cite{endperiodic}). An endperiodic map $f\sm{L}$ is \emph{spA$^+$} if there exists a circular pseudo-Anosov flow $\varphi$ on a hyperbolic 3-manifold $M$ and a depth one foliation $\mathcal{F}$ transverse to $\varphi$ such that $L$ is a depth one leaf of $\mathcal{F}$ and $f$ is a power of the first return map induced by $\varphi$.
  \end{mydef}

Thus, given a spA map $f\sm{L}$, it is isotopic to a spA$^+$ map $f^+$. Both $f$ and $f^+$ minimize points of period $n$ for all $n$ over all endperiodic homotopic maps by Theorem 6.1 of \cite{endperiodic}, so for all $n$, $f$ and $f^+$ have the same number of points of period $n$. Therefore, for our purposes we may as well replace $f$ with $f^+$ and assume $f$ is spA$^+$ from here on, although this will not affect any of the arguments. The only reason to do this is to show that we don't need to worry about dynamic blowups for the purposes of this paper. 

\subsubsection{Periodic points and leaves}

A spA map $f\sm{L}$ comes with a collection of data. This includes the 3-manifold $M$, the depth one foliation $\mathcal{F}$, and the pseudo-Anosov suspension flow $\varphi$ as in Definition \ref{def:spa+}. Let $\mathcal{F}_0$ and $\mathcal{F}_1$ be the sets of depth 0 and depth 1 leaves of $\mathcal{F}$ respectively. We also have $N$, the compactified mapping torus of $f$, which we can identify with the connected component of $M\hspace{-0.5mm}\setminus\hspace{-2mm}\setminus \mathcal{F}_0$ containing $L$, where $M\hspace{-0.5mm}\setminus\hspace{-2mm}\setminus \mathcal{F}_0$ is $M$ cut along $\mathcal{F}_0$ (see Section 1.3.1 of \cite{primer} for a definition). As $\varphi$ is a pseudo-Anosov flow on $M$, we have the invariant \emph{stable} and \emph{unstable foliations}, denoted $W^s$ and $W^u$. We can define foliations on $L$ and $N$ by $\mathcal{W}^{u/s}=W^{u/s}\cap L$ and $W_N^{u/s}=W^{u/s}\cap N$. 

Let $\ell$ be a half leaf of $\lfol$ based at a point $p\in L$. We say $\ell $ is \emph{periodic} if $p$ is a periodic point of $f$. Given an end $\mathcal{E}$ of $L$, we say $\ell$ \emph{escapes the end} $\mathcal{E}$ or $\ell$ is an \emph{escaping half leaf} if for all neighborhoods $U_\mathcal{E}$ of $\mathcal{E}$, all but a compact portion of $\ell$ is contained in $U_\mathcal{E}$. Otherwise, we say $\ell$ is \emph{recurrent}. Note that $\ell$ is recurrent if and only if it intersects some compact set $K\subseteq L$ infinitely often, and that this compact set may be taken to be a closed geodesic $\alpha$. By Proposition 4.6 of \cite{endperiodic} (see also Theorem C of \cite{fenley2009geometry}), the standard hyperbolic metric on $L$ can be chosen such that half leaves of $\widetilde{\mathcal{W}}^{u/s}$ are uniformly quasi-geodesic, so a lift $\widetilde{\ell}$ of $\ell$ to $\widetilde{L}$ has a well-defined endpoint on $\hb$. 

Let $p\in L$ be a fixed point of $f\sm{L}$, and let $\widetilde{f}\sm{\widetilde{L}}$ be a lift fixing $\widetilde{p}$, a lift of $p$. Then we can relate the dynamics of $\widetilde{f}$ on $\hb$ with the endpoints of lifts of half leaves of $\lfol$ based at $p$. Say a lift $\widetilde{g}\sm{\widetilde{L}}$ of an endperiodic map $g\sm{L}$ has \emph{multi sink-source dynamics} if $\widetilde{g}\mid_{\hb}$ has at least four fixed points, and the fixed points alternate between attracting and repelling with regard to the action of $\widetilde{g}$ on $\hb$. Then by Theorem 4.1 of \cite{endperiodic}, for $f\sm{L}$ spA with lift $\widetilde{f}\sm{\widetilde{L}}$ with fixed point $\widetilde{p}$ a lift of $p\in L$, there exists some $k\geq 1$ such that $\widetilde{f}^k$ acts on $\hb$ with multi sink-source dynamics. Furthermore, the fixed points of $\widetilde{f}^k$ on $\hb$ are exactly the endpoints of lifts of half leaves of $\lfol$ based at $\widetilde{p}$. Thus, if $\widetilde{f}$ acts on $\hb$ with multi sink-source dynamics, we can call a fixed point $z\in\hb$ \emph{escaping} or \emph{recurrent} based on the behavior of the corresponding half leaf.

\subsection{The Lefschetz-Hopf theorem}
Our proof of Theorem A relies on the Lefschetz-Hopf theorem; we recall the necessary background here (see \cite{brown1971lefschetz} for more details). Given a map $h\sm{S^n}$ from the $n$-sphere to itself, we have an induced homomorphism $h_*\sm{H_n(S^n)}$. Since $H_n(S^n)\cong \Z$, the map $h_*$ is of the form $h_*(x)=kx$ for some $k\in\Z$. We define the \emph{degree} of $h$ by $deg(h)=k$. 

Now suppose $j\sm{\R^n}$ is a map with an isolated fixed point $x_0\in \R^n$. Let $B$ be a closed ball centered at $x_0$ and containing no other fixed points of $j$, and define a map $h(x)=\frac{x-j(x)}{||x-j(x)||}$
from $\partial B$ to the unit sphere. We define the \emph{fixed point index} of $j$ at $x_0$ by $I(j,x_0)=deg(h)$. Note that if we have a homotopy $\hty{H}{\R^n}$ such that $H_0=j$, $H_1=j'$, $H_t(x_0)=x_0$, and $x_0$ is the only fixed point of $H_t$ in $B$ for all $t$, then $I(H_t,x_0)$ is well-defined and continuous in $t$, and therefore must be constant, so $I(j,x_0)=I(j',x_0)$.

We can extend the definition of the index of a fixed point to a map $j\sm{M}$ on a manifold without boundary by restricting our attention to a chart. If $\partial M\neq\varnothing$ and $x_0\in\partial M$ is a fixed point of $j$, we define $I(j,x_0)=\frac{1}{2}I(Dj, x_0)$, where $Dj\sm{DM}$ is the double map, defined to be $j$ on each copy of $M$ in $DM=M\cup_{id_{\partial M}}M$.

The Lefschetz-Hopf theorem states that for $j\sm{X}$ a map on a triangulable space $X$ with finitely many fixed points, we have 
\[\sum_{x_0\in\text{Fix}(j)}I(j,x_0)=\Lambda_j,\] where $\Lambda_j$ is the \emph{Lefschetz number} of $j$, defined by
\[\Lambda_j=\sum_{i\geq 0}(-1)^i tr(j_*\sm{H_i(X,\Q)}).\] 

We will be concerned with that case that $j$ is an orientation-preserving homeomorphism of the closed disk, in which case $\Lambda_j=1$.

\section{The proof of Theorem A}
Let $f\sm{L}$ be spA, and let $g\sm{L}$ be homotopic to $f$ via a homotopy $H\colon L\times I\rightarrow L$ with $H_0=f$ and $H_1=g$. Let $\widetilde{f}\sm{\overline{\Hyp^2}}$ be a lift of $f$, and let $\widetilde{H}$ be a lift of $H$ with $\widetilde{H}_0=\widetilde{f}$. Then $\widetilde{g}=\widetilde{H}_1$ is a lift of $g$ with $\widetilde{f}\mid_{\hb}=\widetilde{g}\mid_{\hb}$. Suppose $p\in L$ is fixed by $f^n$ and $q\in L$ is fixed by $g^n$ for some $n$. We say $p$ and $q$ are \emph{Nielsen equivalent} if there exist lifts $\widetilde{p}$ and $\widetilde{q}$ of $p$ and $q$ to $\Hyp^2$ such that $\widetilde{f}^n(\widetilde{p})=\widetilde{p}$ and $\widetilde{g}^n(\widetilde{q})=\widetilde{q}$.

To prove our main theorem, we'd like to show that for every point of period $n$ of $f$, there is a Nielsen equivalent point of period $n$ of $g$, thus getting an injective map from points of period $n$ of $f$ to points of period $n$ of $g$. To do this, we could use the fact that for compatible lifts $\widetilde{f},\widetilde{g}\sm{\overline{\Hyp^2}}$ of $f$ and $g$, the lifts agree on $\hb$ and (up to passing to a power) act on $\hb$ with multi sink-source dynamics. If we can show that each fixed point of $\widetilde{g}^n$ on $\hb$ has index $\frac{1}{2}$, then $\widetilde{g}^n$ will have to have a fixed point of negative index in $\Hyp^2$. However, to show that a fixed point $z\in\hb$ of $\widetilde{g}^n$ has index $\frac{1}{2}$, we'll need some control over the dynamics of $\widetilde{g}^n$ in $\Hyp^2$ near $z$. To get this control, we'll focus on recurrent fixed points $z\in\hb$. The following lemma shows that restricting our attention in this way won't affect the limit in the main theorem, as we are only ignoring finitely many half leaves.

\begin{lemma} \label{lem:rec}
    Only finitely many periodic half leaves of $f$ escape an end.
\end{lemma}

\begin{proof}
    Let $p\in L$ be a periodic point of $f$, and let $\ell$ be a half leaf at $p$ that escapes an end. Suppose that $\ell$ is a leaf of the unstable foliation; the case that $\ell$ is a leaf of the stable foliation is similar. Since $p$ is periodic, it is contained in a closed orbit $\gamma$ of the flow $\varphi$. Let $H_0$ be the periodic leaf of $W^u$ containing $\gamma$. Cut $H_0$ along $\partial_+ N$ and the periodic orbits of $\varphi$. Then based on the description of these pieces given in Lemma 4.4 of \cite{endperiodic}, exactly one of these components contains $\gamma$; call this one $C$. Since $\gamma$ is one of the boundary components of $C$, we know $C$ is of type (1) or type (2), as in Figure \ref{fig:types} (reproduced from Figure 7 of \cite{endperiodic}); we will show it must be of type (1). 

    \begin{figure}[h]
        \centering
        \includegraphics[width=0.6\linewidth]{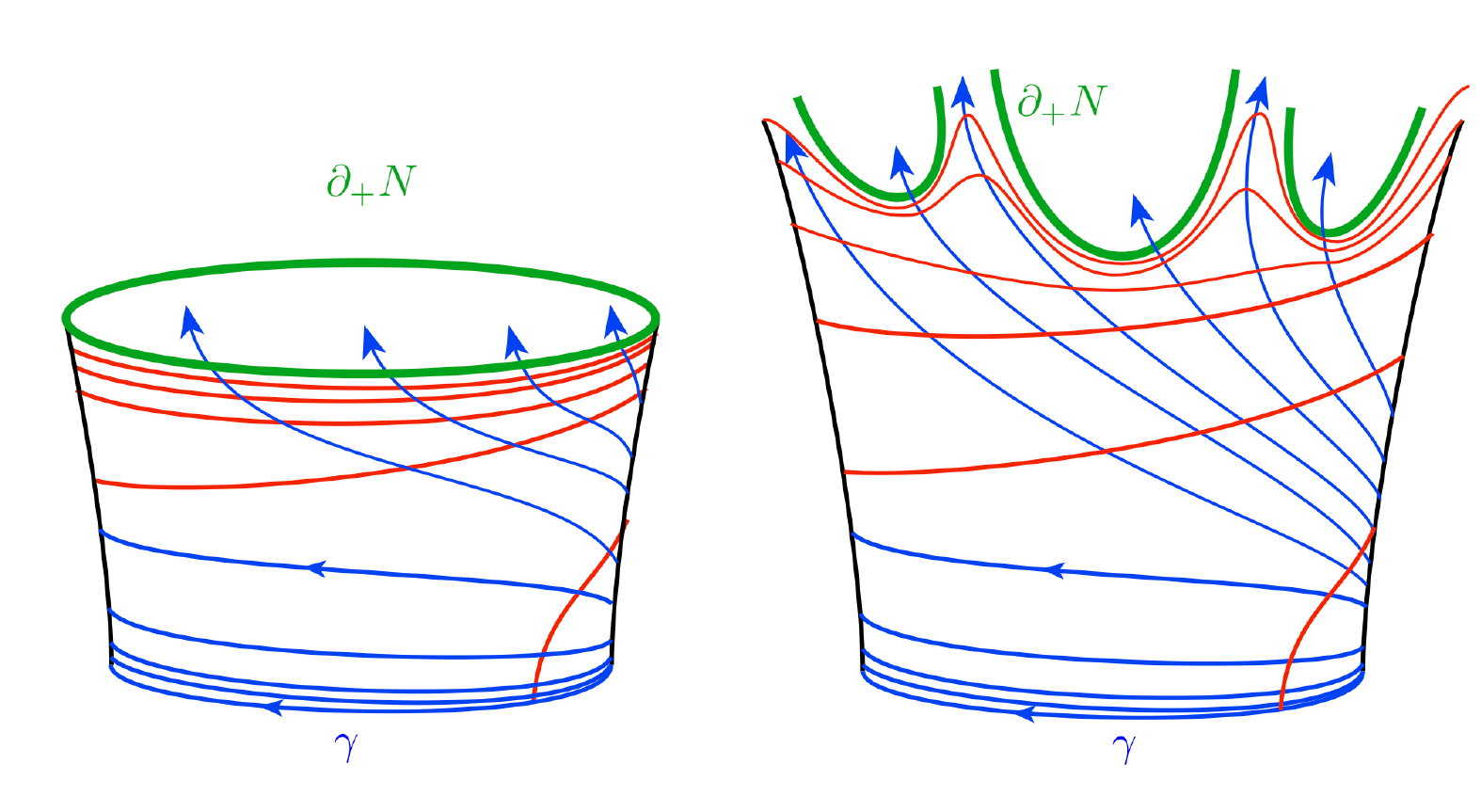}
        \caption{The options for $C$ in Lemma \ref{lem:rec}. Type (1) is on the left, and type (2) is on the right. In blue are flow lines of $\varphi$, and in red is $\ell$.}
        \label{fig:types}
    \end{figure}
    
    Let $U$ be a collar neighborhood of $\partial_+ N$ in $N$. Intersecting with $C$ gives us a collar neighborhood $U\cap C$ of $\partial_+ N\cap C$ in $C$. Considering the way $\ell$ sits inside $C$, we see by Lemma 4.5 of \cite{endperiodic} that either $C$ has a single boundary component from $\partial_+ N$, which is a closed curve, and $\ell$ is eventually contained inside any neighborhood of this boundary component, or $C$ has some number of noncompact boundary components from $\partial_+ N$, and $\ell=C\cap L$ accumulates on all of them. 

     Let $D$ be the component of $\partial_+ N$ corresponding to the end out which $\ell$ escapes. Then for any neighborhood $U'$ of $D$ in $N$, $\ell$ is entirely contained within $U'$ past some point. Thus $\ell$ is eventually contained in $U\cap C$. This implies that $C$ is of type (1) in Lemma 4.5 of \cite{endperiodic}, i.e. $C\cap \partial_+ N$ is a single closed curve. 

     The above argument gives us a map from leaves of $f$ starting at periodic points and escaping an end to closed leaves of the foliation $W^u\cap{\partial_+N}$. We claim that this map has finite image and is injective. 

     To see that the image is finite, note that, $W^u\cap{\partial_+N}$ has finitely many pairwise nonhomotopic closed leaves. Also note that two homotopic closed leaves cannot be in the image of this function, as the two corresponding periodic orbits would then be homotopic, and no two distinct closed orbits of a pseudo-Anosov suspension flow are homotopic. This follows from the fact that periodic points of pseudo-Anosovs are unique in their Nielsen classes, see Lemma 3.1 of \cite{handelthurston}.

     To see injectivity, suppose leaves $\ell$ and $\ell'$ from periodic points $p$ and $p'$ in leaves $H_0$ and $H_0'$ were sent to the same closed curve $\alpha$ in $\partial_+N$. Then since leaves of $W^u$ are disjoint, $H_0=H_0'$, and again we'd have two homotopic closed orbits, a contradiction.
    \end{proof}

\subsection{Recurrent fixed points}\label{subsec:recurrent}

Now we'll focus our attention on recurrent fixed points of $f$. Let $\widetilde{f}\sm{\widetilde{L}}$ be a lift of $f$, and $z\in\hb$ be a recurrent fixed point of $\widetilde{f}$ with associated periodic half leaf $\ell\subseteq \mathcal{W}^{u/s}$ and lift $\widetilde{\ell}$ of $\ell$. Let $\alpha\subseteq L$ be a closed geodesic that intersects $\ell$ infinitely often. Each intersection of $\ell$ with $\alpha$ defines a lift of $\alpha$ intersecting $\widetilde{\ell}$; label these lifts $\{\widetilde{\alpha}_i\}_{i\in\N}$, ordered from the perspective of moving along $\ell$. The following lemma is a modification of Lemma 1.1 from \cite{handelthurston} and gives us a neighborhood basis for $z$ using $\{\widetilde{\alpha}_i\}_{i\in\N}$.

\begin{lemma}\label{lem:nbd}
    Suppose $z\in\partial_\infty\Hyp^2$ is recurrent. Let $\ell,$ $\alpha$, and $\{\widetilde{\alpha}_i\}_{i\in\N}$ be as above. Let $N_i$ be the component of $\overline{\Hyp^2}\setminus \widetilde{\alpha}_i$ containing $z$. Then $\{N_i\}_{i\in\N}$ forms a neighborhood basis for $z$.
\end{lemma}

\begin{proof}
    There is a uniform positive lower bound to the distance between two consecutive lifts $\widetilde{\alpha}_i$ and $\widetilde{\alpha}_{i+1}$ since they are lifts of the same closed geodesic $\alpha$, implying $\Hyp^2\cap\bigcap_i N_i=\varnothing$. Since $\widetilde{\ell}\cap\widetilde{\alpha}_i\rightarrow z$ as $i\rightarrow\infty$, we must have $\bigcap_i N_i=z$.
\end{proof}

Our next two lemmas are inspired by parts of the proof of Lemma 3.1 in \cite{handelthurston}. In their setting they are studying compact surfaces; we need to be careful to do everything in a lift of a compact subsurface to generalize these techniques to our setting. 

The goal of the following lemma is roughly to show that close enough to $z\in\hb$, $\widetilde{f}$ moves points far. In Lemma \ref{lem:index}, we will find a homotopy $H\colon \overline{\Hyp^2}\times I\rightarrow \overline{\Hyp^2}$ taking $\widetilde{g}$ to a new map $g'$. To show that $g'$ has the same index as $\widetilde{g}$ at $z$, we need to show that in a neighborhood of $z$, each $H_t$ has no fixed points other than $z$. We do this by showing via Lemma \ref{lem:K} that $H_t$ first moves points far, and then moves them back only a little, so it can't create any new fixed points.

\begin{lemma}\label{lem:K}
    Given any $K>0$ and compact subsurface $S\subseteq L$, there exists an $i$ such that for all $p\in N_i\cap \pi^{-1}(S)$, $d(p,f(p))>K$.
\end{lemma}

\begin{proof}
    Suppose towards contradiction that we have a sequence of points $p_i\in\pi^{-1}(S)$ such that $p_i\rightarrow z$ and $d(p_i,f(p_i))\leq K$ for some $K>0$. Since $S$ is compact, each $p_i$ is a uniformly bounded distance away from a lift of a single point $q_0\in S$, so we may assume up to changing the constants that each $p_i$ is a different lift of  $q_0$. Thus, there are $\gamma_i\in\pi_1(L)$ such that $p_i=\gamma_i(p_1)$. Then 
    \begin{align*}
        K&\geq d(p_i,f(p_i))\\
        &=d(\gamma_i(p_1),f\gamma_i(p_1))\\
        &=d(p_1,\gamma_i^{-1}f\gamma_i(p_1)). 
    \end{align*}
    Since the set of images of $p_1$ under all lifts of $f$ is discrete (they are all lifts of the point $f(q_0)$), we must have $\gamma_i^{-1}f\gamma_i=\gamma_j^{-1}f\gamma_j$ for some $i\neq j$. Then $\gamma_j\gamma_i^{-1}f=f\gamma_j\gamma_i^{-1}$, which contradicts the fact that $f$ commutes with no nontrivial deck transformations by Lemma \ref{lem:commute}. Thus, $z$ has an open neighborhood $N$ such that every point in $N\cap \pi^{-1}(S)$ is moved at least $K$ by $f$. By Lemma \ref{lem:nbd}, we can take $N$ to be $N_i$ for some $i$.
\end{proof}

\begin{lemma}\label{lem:index}
    Let $p$ be a point of period $n$ of $f$ with no escaping half leaves, let $\widetilde{f}$ be a lift of $f$ to $\Hyp^2$ and $\widetilde{p}$ a compatible lift of $p$. Suppose $g\simeq f$, and let $\widetilde{g}$ be a compatible lift of $g$. If $\widetilde{g}$ has finitely many fixed points, $\widetilde{f}$ acts on $\partial_\infty\Hyp^2$ with multi sink-source dynamics, and $z\in\partial_\infty\Hyp^2$ is fixed by $\widetilde{f}$, then $I(\widetilde{g},z)=\frac{1}{2}$.
\end{lemma}

\begin{proof}
Since $\widetilde{f}$ acts on $\partial_\infty\Hyp^2$ with multi sink-source dynamics, $z$ is isolated as a fixed point of $\widetilde{f}$ and is either a sink or a source; suppose it is a sink. We know $\widetilde{f}$ and $\widetilde{g}$ agree on $\partial_\infty\Hyp^2$, so $z$ is a sink for $\widetilde{g}$ as well. Note that this is with respect to the action of $\widetilde{g}$ on $\partial_\infty\Hyp^2$, not $\overline{\Hyp^2}$. Our strategy is to find a neighborhood $U_2$ of $z$ and a map $g'\sm{L}$ homotopic to $\widetilde{g}$ through maps with no fixed points other than $z$ in $U_2$  such that $g'(U_2)\subsetneq U_2$.

 Let $\gamma$ be a geodesic ray in $\Hyp^2$ going to $z$, and let $\alpha$ be a closed geodesic in $L$ that $\gamma$ crosses infinitely many times. Let $H\colon I\times L\rightarrow L$ be a homotopy such that $H_0=id_L$ and $H_1(g(\alpha))=g_*(\alpha)$, the geodesic representative of $g(\alpha)$. This homotopy can be taken to be supported on some compact subsurface $S\subseteq L$. Since $H$ is compactly supported, no point on $L$ is moved more than some constant $C$ during the homotopy, i.e. the path traced out by any point during the homotopy has length less than $C$.

 Since $\widetilde{g}$ has finitely many fixed points, we can find a neighborhood $N_i$ as in Lemma \ref{lem:nbd} such that $N_i\setminus z$ contains no fixed points of $\widetilde{g}$ and such that $\widetilde{g}$ moves points in $\partial_\infty\Hyp^2\cap N_i$  towards $z$. Let $U_1=N_i$. Now pick $N_j$ with $j>i$ such that $\widetilde{g}(N_j)\subseteq U_1$, and all points in $N_j\cap\pi^{-1}(S)$ are moved more than $3C$ by $\widetilde{g}$, which is possible by Lemma \ref{lem:K}. Let $U_2=N_j$. Relabel the set $\{\widetilde{\alpha}_i\}_i$ so $\widetilde{\alpha}_1$ and $\widetilde{\alpha}_2$ border $U_1$ and $U_2$ respectively. See Figure \ref{fig:index}. 

 \begin{figure}[h]
     \centering
     \includegraphics[width=0.45\linewidth]{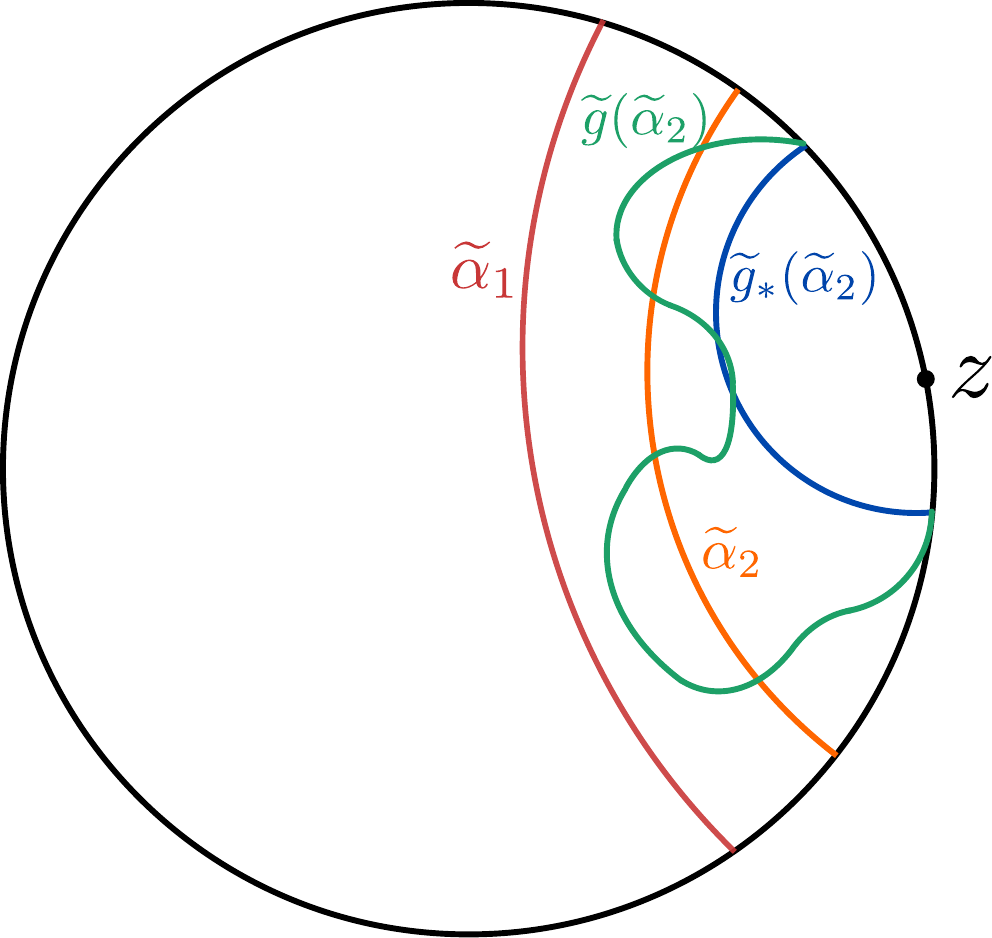}
     \caption{The setup in Lemma \ref{lem:index}.}
     \label{fig:index}
 \end{figure}

 Let $\widetilde{S}$ be the connected component of $\pi^{-1}(S)$ containing $\widetilde{\alpha}_2$, and define a homotopy $\widetilde{H}\colon I\times\overline{\Hyp^2}\rightarrow\overline{\Hyp^2}$ as follows. On $\overline{\Hyp^2}\setminus \widetilde{S}$, let $\widetilde{H}(t,x)=x$. On $\widetilde{S}$, let $\widetilde{H}$ be the unique lift $H'$ of $H\mid_{I\times S}$ such that ${H'}_0=id_{\widetilde{S}}$. Since ${H'}_0\mid_{\partial\widetilde{S}}=id_{\partial\widetilde{S}}$, and ${H'}_t$ is a lift of $id_{\partial S}$ on each component of $\partial\widetilde{S}$, we must have that $H'_t\mid_{\partial \widetilde{S}}=id_{\partial\widetilde{S}}$, as distinct components of $\partial\widetilde{S}$ are positive distance apart, and lifts of $id_{\partial S}$ on a single component of $\partial\widetilde{S}$ differ by hyperbolic isometries of positive translation length. Thus, we have that $\widetilde{H}$ is a well defined homotopy on all of $\overline{\Hyp^2}$.

 Since $\pi\circ\widetilde{H}_1\circ\widetilde{g}(\widetilde{\alpha}_2)=g_{*}(\alpha)$ and $\widetilde{H}_1\circ\widetilde{g}(\widetilde{\alpha}_2)\subset \widetilde{S}$, we must have $\widetilde{H}_1\circ\widetilde{g}(\widetilde{\alpha}_2)=\widetilde{g}_{*}(\widetilde{\alpha}_2)$, the geodesic representative of $\widetilde{g}(\widetilde{\alpha}_2)$. Let $g'=\widetilde{H}_1\circ \widetilde{g}$.  Since $\widetilde{g}$ moves points in $N_2\cap\widetilde{S}$ at least $3C$, and $\widetilde{H}_t$ moves points in $\widetilde{S}$ no more than $C$ for all $t\in[0,1]$, $\widetilde{H}_t\circ \widetilde{g}$ has no fixed points in $N_2\cap\widetilde{S}$ for all $t$. Since $\widetilde{g}$ has no fixed points in $N_2\setminus{z}$ and $\widetilde{H}_t$ acts by the identity on $N_2\setminus\widetilde{S}$, we have that $\widetilde{H}_t\circ \widetilde{g}$ has no fixed points on all of $N_2\setminus{z}$, but does fix $z$. Thus, $I(\widetilde{g},z)=I(g',z)$. 
    
 To compute $I(g',z)$, note that the endpoints of $\widetilde{g}_*(\widetilde{\alpha_2})$ are in $\partial_\infty\Hyp^2\cap U_2$, so $g'(N_2)\subsetneq N_2$, and $z$ is a sink for the action of $g'$ on all of $\overline{\Hyp^2}$. Thus, $I(g',z)=\frac{1}{2}$. 
\end{proof}

\begin{theorem}\label{thm:fewer}
    Let $f\sm{L}$ be spA. For some $N\in \N$, every $g\simeq f$ has no fewer points of period $n$ than $f$ for all $n> N$.
\end{theorem}

\begin{proof}
    By Lemma \ref{lem:rec}, only finitely many periodic points of $f$ have half leaves that escape an end. Let $N$ be the highest period of such a point. Now let $p$ be any periodic point of $f$ of period $n>N$. Our strategy is to find a Nielsen equivalent point $q$ that has period $n$ under $g$. Since distinct points of period $n$ of $f$ aren't Nielsen equivalent by Lemma 4.8 of \cite{endperiodic}, this will define an injective map from points of period $n$ of $f$ to points of period $n$ of $g$.

    So, let $\widetilde{f}$ be a lift of $f$, let $\widetilde{p}$ be a compatible lift of $p$, and $\widetilde{g}$ a compatible lift of $g$, so that $\widetilde{f}$ and $\widetilde{g}$ agree on $\partial_\infty\Hyp^2$. By Theorem 4.1 of \cite{endperiodic}, we know that for some smallest $k$, $\widetilde{f}^k$ acts on $\partial_\infty\Hyp^2$ with multi sink-source dynamics. In the case that $k\geq2$, the argument in the proof of Theorem 6.1 from \cite{endperiodic} carries over. In particular, $\widetilde{f}$ and $\widetilde{g}$ have no fixed points on $\hb$, so by the Lefschetz-Hopf theorem applied to the disk $\overline{\Hyp^2}$, $\widetilde{g}$ must have a fixed point in $\Hyp^2$.
    
    Now assume $k=1$.  By Lemma \ref{lem:index}, all the fixed points of $\widetilde{g}$ on $\partial_\infty\Hyp^2$ have index $\frac{1}{2}$. Since $\chi(S^2)=2$, the double of $\widetilde{g}$ must have a fixed point of negative index, so $\widetilde{g}$ has some fixed point $\widetilde{q}$ in $\Hyp^2$. All that remains is to show $\widetilde{q}$ has period exactly $n$ under $\widetilde{g}$.
    This same situation occurs in the corresponding proof for pseudo-Anosovs in the finite-type setting, and the same argument used there as in Theorem 14.20 of \cite{primer} works here.
\end{proof}

This immediately gives us the following corollary, where $\lambda(f)$ is the \emph{growth rate} of $f$, defined in the introduction.

\begin{corollary}\label{cor:growth}
    Let $f\colon L\rightarrow L$ be spA.  Then \[\lambda(f)=\inf_{g\simeq f}\lambda(g), \] where the infimum is taken over all maps $g$ homotopic to $f$.
\end{corollary}

\section{The Handel--Miller case}

\subsection{Handel--Miller theory}\label{subsec:handelmiller}
Let $g\sm{L}$ be an atoroidal endperiodic map with positive and negative escaping sets $\mathcal{U}_+$ and $\mathcal{U}_-$. A $g$-\emph{cycle of ends} is a set \[\{\mathcal{E},g(\mathcal{E}),g^2(\mathcal{E}),\ldots\}\] for $\mathcal{E}$ an end of $L$. Each $g$-cycle of ends is finite, and there are finitely many of them. Let $A_1,\ldots,A_k$ be the set of $g$-cycles of ends, and pick a representative end $\mathcal{E}_i$ in each $A_i$. Fix a single-component tightened $g$-juncture for each $\mathcal{E}_i$ as in Section \ref{subsec:endperiodic}. Let $\mathcal{J}$ be the set of geodesic tightenings of the $g$-orbits of these $g$-junctures. Let $\mathcal{J}_+$ be the set of positive $g$-junctures in $\mathcal{J}$ and $\mathcal{J}_-$ the set of negative $g$-junctures in $\mathcal{J}$. We say a component $\gamma$ of $\mathcal{J}$ \emph{escapes} if the geodesic tightenings of the $g$-orbit of $\gamma$ escape compact sets. Let $\mathcal{X}_\pm$ be the set of nonescaping components of $\mathcal{J}_\pm$. Then we define the \emph{positive} and \emph{negative Handel--Miller laminations} by $\Lambda^+_{HM}=\overline{\mathcal{X}_-}\setminus\mathcal{X}_-$ and $\Lambda^-_{HM}=\overline{\mathcal{X}}_+\setminus\mathcal{X}_+$. For more details on this construction, see Section 4.3 of \cite{cantwell2021endperiodic}.

A \emph{positive principal region} is a connected component of $\mathcal{P}_+=L\setminus\overline{\mathcal{U}}_-$, and a \emph{negative principal region} is a connected component of $\mathcal{P}_-=L\setminus \overline{\mathcal{U}}_+$. By work of Cantwell, Conlon, and Fenley, there are finitely many principal regions, and each is the interior of a finite sided ideal polygon (Theorem 6.5 of \cite{cantwell2021endperiodic}). A \emph{Handel--Miller map} is a homeomorphism $h\sm{L}$ such that $h$ is isotopic to $g$ and $h$ preserves $\Lambda^+_{HM}$, $\Lambda^-_{HM}$, and a choice of tightened $g$-junctures for each end of $L$. Given a Handel--Miller map $h$, we have $\Lambda^+_{HM}=\partial\mathcal{U}_-$ and $\Lambda^-_{HM}=\partial\mathcal{U}_+$. Every atoroidal endperiodic $g\sm{L}$ has a Handel--Miller representative by Theorem 4.54 of \cite{cantwell2021endperiodic}.

Fixing such a Handel--Miller map $h$, call a half leaf $\ell$ of $\Lambda^\pm_{HM}$ based at a point $q$ \emph{periodic} if $q$ is a periodic point of $h$. We say $\ell$ \emph{escapes an end} $\mathcal{E}$ of $L$ if $\ell\cup(L\setminus U_\mathcal{E})$ is compact for every neighborhood $U_\mathcal{E}$ of $\mathcal{E}$. Otherwise, we say $\ell$ is \emph{recurrent}. If $\ell$ is recurrent, it intersects some closed geodesic $\alpha\subset L$ infinitely often. All but finitely many periodic half leaves are recurrent (see Theorem 6.27 of \cite{cantwell2021endperiodic}). If a point $z\in\hb$ is an endpoint of a lift $\ell$ of a half leaf $\ell\subset \Lambda^\pm_{HM}$, we say $z$ is \emph{recurrent} whenever $\ell$ is. 

We will need the following proposition from \cite{endperiodic}.

\begin{prop}\label{prop:8.2}(Proposition 8.2 of \cite{endperiodic}) Let $\widetilde{h}\sm{\widetilde{L}}$ be a lift of $h$ such that $\widetilde{h}^k$ fixes a point $\widetilde{q}\in\widetilde{L}$ for some $k$. Then either $\widetilde{q}$ is in the closure of a lift of a principal region, or $\widetilde{q}$ is the unique fixed point of $\widetilde{h}^k$ in $\widetilde{L}$, and $\widetilde{q}=\widetilde{\ell^+}\cap\widetilde{\ell^-}$ for $\widetilde{\ell^\pm}$ the unique leaves of $\widetilde{\Lambda^\pm}_{HM}$ through $\widetilde{q}$. In the latter case, for some $p$, $\partial\ell^\pm$ are exactly the fixed points of $\widetilde{h}^{kp}$ on $\hb$.   
\end{prop}

In the second case, $\widetilde{h}^{kp}$ acts on $\hb$ with multi sink-source dynamics by Lemma 4.9 of \cite{endperiodic}.

\subsection{The proof of Theorem C}

Throughout, let $h\sm{L}$ be an atoroidal Handel--Miller map. To prove Theorem C, we need to modify Lemmas \ref{lem:nbd}, \ref{lem:K}, and \ref{lem:index} to apply in this setting. Lemma \ref{lem:nbd} holds without modification in the proof for points $z\in\hb$ that are recurrent with respect to the Handel--Miller laminations. Lemma \ref{lem:K} holds when we replace the spA map $f$ with the atoroidal Handel--Miller map $h$. For Lemma \ref{lem:index}, we need to add one hypothesis.

\begin{lemma}\label{hindex}
  Let $q$ be a point of period $n$ of $h$ with no escaping half leaves, let $\widetilde{h}$ be a lift of $h$ to $\Hyp^2$ and $\widetilde{q}$ a compatible lift of $q$. Assume further that $\widetilde{q}$ is the unique fixed point of $\widetilde{h}^n$ in $\Hyp^2$. Suppose $g\simeq h$, and let $\widetilde{g}$ be a compatible lift of $g$. If $\widetilde{g}$ has finitely many fixed points, and $z\in\partial_\infty\Hyp^2$ is fixed by $\widetilde{h}$, then $I(\widetilde{g},z)=\frac{1}{2}$.   
\end{lemma}

\begin{proof}
    We have that $\widetilde{h}$ acts on $\hb$ with multi sink-source dynamics by Proposition \ref{prop:8.2}. The rest of the proof follows exactly as in the proof of Lemma \ref{lem:index}.
\end{proof}

\begin{theorem}\label{thm:hm}
     Let $h\sm{L}$ be an atoroidal Handel--Miller map with stable and unstable Handel--Miller laminations $\Lambda^+_{HM}$ and $\Lambda^-_{HM}$. For some $N\in\N$, every $g\simeq h$ has no fewer points of period $n$ than the core dynamical system of $h$ for all $n>N$.
\end{theorem}

    \begin{figure}[h]
        \centering
        \includegraphics[width=0.4\linewidth]{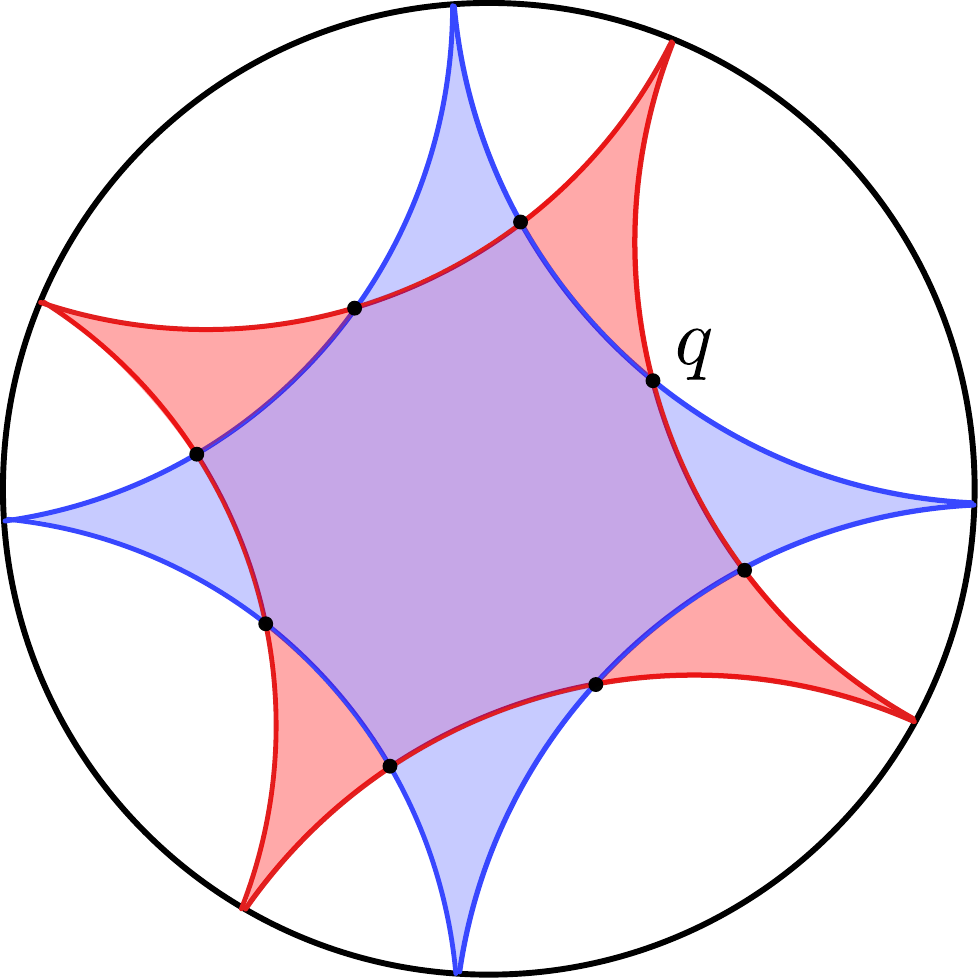}
        \caption{The red region is the closure of a lift of a positive principal region, and the blue region is the closure of a lift of a negative principal region. The point $q$ must lie on one of the vertices, shown in black, of the polygonal intersection of these two regions.}
        \label{fig:principal}
    \end{figure}

\begin{proof}
    Our main claim is that there exists some $N\in\N$ such that if $q\in\Lambda^+_{HM}\cap\Lambda^-_{HM}$ is a periodic point of $h$ of period $n>N$ contained in leaves $\ell\pm\subset\Lambda^\pm_{HM}$, then for some lift $\widetilde{q}$ of $q$ fixed by $\widetilde{h}^n$, where $\widetilde{h}$ is a lift of $h$, we have that $\widetilde{q}$ is the unique fixed point in $\widetilde{L}$ of $\widetilde{h}^n$, and at least one of $\ell^+$and $\ell^-$ is recurrent. If $\widetilde{q}$ is not the unique fixed point of $\widetilde{h}^n$, then $\widetilde{q}$ lies in the closure of a lift of a principal region by Proposition \ref{prop:8.2}, and must be as in Figure \ref{fig:principal}. As there are only finitely many principal regions of $h$, each with only finitely many boundary leaves, there can only be finitely many such points $q$. Similarly, all but finitely many periodic half leaves are recurrent. Thus, we can let $N$ be the maximum period among periodic points that lie on the intersection of boundary leaves of principal regions and each of whose leaves are escaping.

   The rest of the proof follows exactly the same as the proof of Theorem \ref{thm:fewer}, using our modified versions of Lemmas \ref{lem:nbd}, \ref{lem:K}, and \ref{lem:index}. In using the argument from Theorem 14.20 of \cite{primer} to show that the fixed point of $\widetilde{g}^n$ has period exactly $n$, we use the hypothesis that $\widetilde{h}^n$ has a unique fixed point.
\end{proof}

\section{Sharpness of results}\label{sec:ex}
One might hope for versions of Theorems \ref{thm:fewer} and \ref{thm:hm} that hold for all periods and do not require the constant $N$. This is what we see in the case of pseudo-Anosovs on finite-type surfaces, but in fact it does not hold in our setting. We will demonstrate this by producing a homeomorphism $j$ without fixed points, homotopic to an endperiodic map, but whose spA and Handel--Miller representatives each have at least one fixed point. In particular, we will first construct an endperiodic homeomorphism $g$ whose spA and Handel--Miller representatives must have fixed points. We will then modify $g$ by an isotopy to get the map $j$, which is fixed point free.

\subsection{Constructing the endperiodic map $g$}\label{subsec:exg}

Let $C$ be the portion of $\R^2$ bounded by the hyperbolas $xy=1$ and $xy=-1$ and containing the origin. Consider the map \[A=\begin{pmatrix}
    2&0\\0&\frac{1}{2}
\end{pmatrix},\] which maps $C$ to itself.

\begin{figure}[ht]
    \centering
    \includegraphics[width=0.9\linewidth]{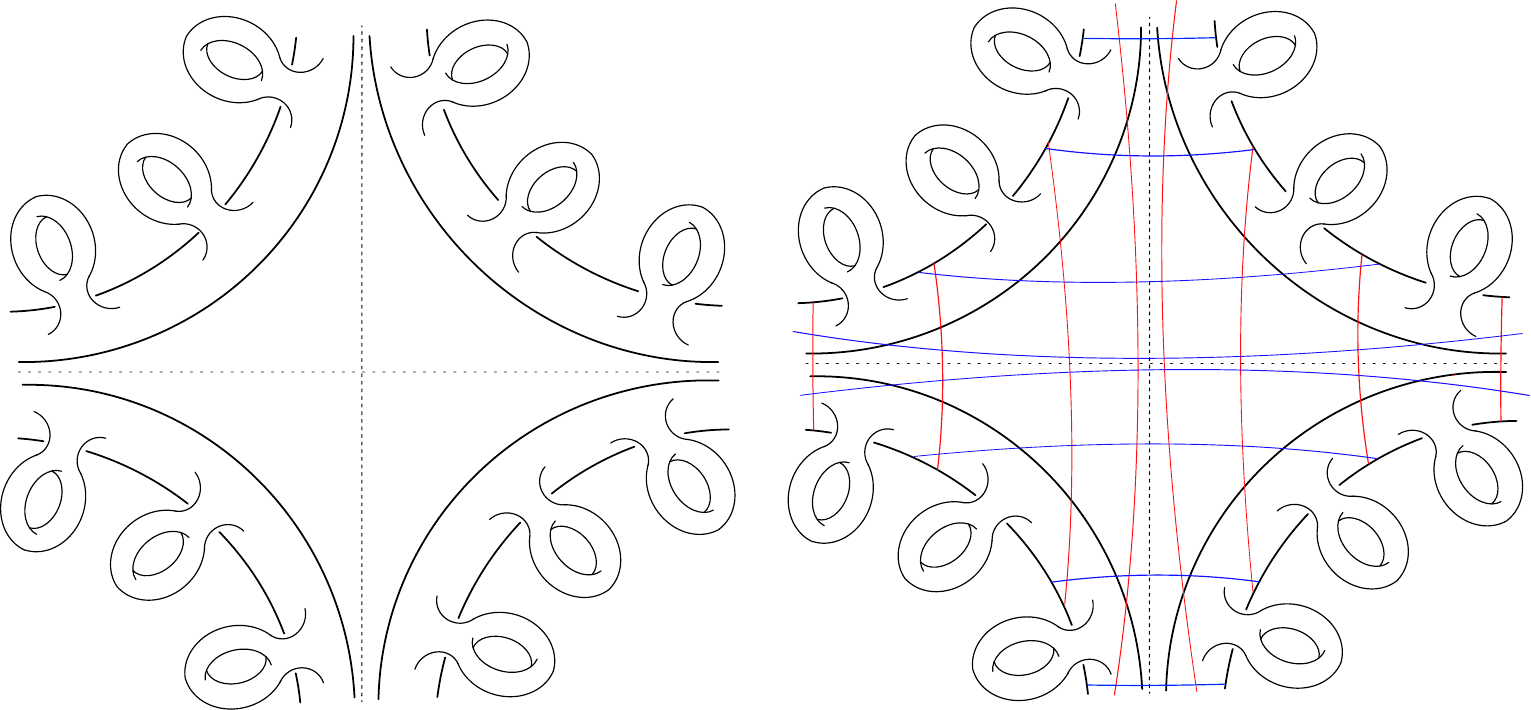}
    \caption{On the left is the surface $S$. On the right is a system of $g$-junctures on $L$. The arcs are drawn on $S$ and should be doubled to get closed curves on $L$. Positive $g$-junctures are drawn in red, and negative $g$-junctures are drawn in blue.}
    \label{fig:Swwofols}
\end{figure}

Attach constant-width infinite strips to each of the four sides of $C$, and extend the action of $A$ across the strips so that they move following the boundary of $C$, and call the extended map $A'$. Choose a nonempty $A'$-invariant set of disks in each of the four strips, and attach handles to these disks to create a new surface $S$, as illustrated in Figure \ref{fig:Swwofols}. Now we can define a map $\hat{A}\sm{S}$ that acts as $A'$ off the handles and just moves the handles with the rest of the surface. 

Note that the doubled surface $DS$ is a boundaryless infinite-type surface with four ends, each of which is accumulated by genus. Let $L=DS$, and let $g=D\hat{A}$, the doubled map. Then since $g\sm{L}$ is endperiodic, it has some spA representative and some Handel--Miller representative.

\subsection{The spA and Handel--Miller representatives of $g$ have fixed points.}\label{subsec:exfh}
Now we will show that any spA or Handel--Miller representative of $g$ has at least one fixed point. Consider the system of $g$-junctures shown in Figure \ref{fig:Swwofols}. Iterating these under positive and negative powers of $g$ shows that the positive Handel--Miller lamination $\Lambda_{HM}^+$ is the union of the two copies of the $x$-axis, and $\Lambda_{HM}^-$ is the union of the two copies of the $y$-axis. Since each lamination contains two leaves, label them so that $\Lambda_{HM}^+=\{\lambda_1^+,\lambda_2^+\}$, $\Lambda_{HM}^-=\{\lambda_1^-,\lambda_2^-\}$, and $\lambda_i^+\cap\lambda_i^-\neq\varnothing$ for $i=1,2$. Let $h$ be a Handel--Miller representative of $g$. Then $h$ must fix the two points $0_1=\lambda_1^+\cap\lambda_1^-$ and $0_2=\lambda_2^+\cap\lambda_2^-$.

Let $\widetilde{h}$ be a lift of $h$ to $\widetilde{L}\cong\overline{\Hyp^2}$ fixing lifts $\widetilde{\lambda_1^+}$ and $\widetilde{\lambda_1^-}$ of $\lambda_1^+$ and $\lambda_1^-$. By Proposition 6.9 of \cite{cantwell2021endperiodic}, $\widetilde{h}$ acts on $\hb$ with multi sink-source dynamics, where the fixed points are exactly the endpoints of $\widetilde{\lambda_1^+}$ and $\widetilde{\lambda_1^-}$. Let $f$ be a spA representative of $g$, and let $\widetilde{f}$ be the lift of $f$ to $\overline{\Hyp^2}$ such that $\widetilde{f}$ agrees with $\widetilde{h}$ on $\hb$. Then since $\widetilde{f}$ acts on $\hb$ with multi sink-source dynamics, we have by Theorem 4.1 of \cite{endperiodic} that $\widetilde{f}$ has a fixed point in $\overline{\Hyp^2}$, and thus $f$ has at least one fixed point in $L$.

\subsection{Constructing the fixed point free map $j$}\label{subsec:exj}

Now we will modify $g$ by an isotopy to create a map $j\sm{L}$ with no fixed points. For simplicity, we will define the map only on $S$, and it should be doubled to get the map $j\sm{L}$.

\begin{figure}[ht]
    \centering
    \includegraphics[width=0.5\linewidth]{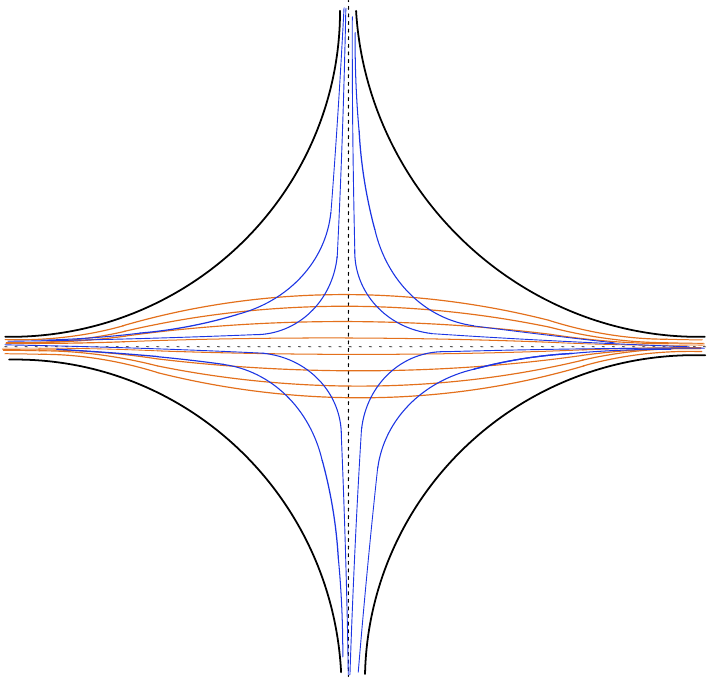}
    \caption{The region $R$ in $C$ is the region foliated in orange. The singular foliation $\mathcal{F}$ of $C$ by hyperbolas of the form $xy=a$ for $a\in[-1,1]$ is shown in blue.}
    \label{fig:R}
\end{figure}

Let $R$ be the subregion of $C$ shown in Figure \ref{fig:R}. Let $\mathcal{F}$ be the singular foliation of $C$ by hyperbolas of the form $xy=a$ for $a\in[-1,1]$, which is invariant under the action of $A$. As in Figure \ref{fig:R}, foliate $R$ by arcs such that each arc meets each leaf of $\mathcal{F}$ at most once (except on the $x$-axis, where the two foliations coincide). Parametrize $R$ as $\R\times[-1,1]$ such that each leaf is $\R\times \{t\}$ for some $t\in[-1,1]$, and a point $x_0$ on the $x$-axis corresponds with $(x_0,0)\in\R\times[-1,1]$. For $t\in[-1,1]$, define $F_t\sm{\R}$ by \[F_t(x)=(1-|t|)(x/2+1)+|t|x.\] Define $F\sm{R}$ by \[F(x,t)=(F_t(x),t).\] Now define the map $j$ first on $R$ by \[j(p)=F\circ g(p),\] and extend this map to the rest of $S$ by \[j(p)=g(p)\] for $p\in S\setminus R$. Note that by construction, $j\simeq g$.

We now show that $j$ has no fixed points. On $S\setminus R$, $g$ has no fixed points, and $j$ agrees with $g$ on this set. On $R$, note that the set of fixed points of $F$ is \[\text{Fix}(F)=\left\{(x,t)\in R\mid x=2\text{ or }|t|=1\right\}.\]

On $R\setminus (\text{Fix}(F)\cup x\text{-axis})$, $j$ moves points off of their leaf in $\mathcal{F}$, and thus can have no fixed points. On $\text{Fix}(F)$, $g$ has no fixed points, so $j$ has no fixed points. Finally, on the $x$-axis, $j(x,0)=(x+1,0)$, so we conclude that $j$ has no fixed points.

\bibliography{bibliography}
\bibliographystyle{alpha}

\end{document}